\documentclass[10pt]{amsart}

\author{Ties Laarakker}
\address{}
\email{p.t.a.laarakker@uu.nl}
\title{Vertical Vafa-Witten invariants}

\usepackage[utf8]{inputenc}

\usepackage{amsmath}
\usepackage{amsfonts}
\usepackage{amssymb}
\usepackage{amsthm}
\usepackage{caption}
\usepackage{enumerate}
\usepackage{hyperref}
\usepackage{mathtools}
\usepackage{mathrsfs}
\usepackage{tikz-cd}
\usepackage{url}

\newcommand{\CC}{\mathbb{C}}
\newcommand{\PP}{\mathbb{P}}
\newcommand{\QQ}{\mathbb{Q}}
\newcommand{\ZZ}{\mathbb{Z}}
\newcommand{\Z}{\mathsf{Z}}

\newcommand{\D}{\mathcal{D}}
\newcommand{\E}{\mathcal{E}}

\newcommand{\F}{\mathcal{F}}
\newcommand{\I}{\mathcal{I}}

\renewcommand{\L}{\mathcal{L}}

\renewcommand{\O}{\mathcal{O}}
\renewcommand{\P}{\mathcal{P}}
\newcommand{\N}{\mathcal{N}}
\newcommand{\SW}{\mathrm{SW}}

\renewcommand{\t}{\mathfrak{t}}

\DeclareMathOperator{\Pic}{Pic}

\DeclareMathOperator{\ch}{ch}

\DeclareMathOperator{\Ext}{Ext}
\DeclareMathOperator{\EExt}{\mathscr{E}\mathnormal{xt}}
\DeclareMathOperator{\Hom}{Hom}
\DeclareMathOperator{\RHom}{\mathnormal{R}\mathscr{H}\mathnormal{om}}
\DeclareMathOperator{\rk}{rk}
\newcommand{\pr}{\mathrm{pr}}
\DeclareMathOperator{\Td}{Td}
\DeclareMathOperator{\tr}{tr}
\newcommand{\vir}{\mathrm{vir}}

\newcommand{\vertical}{\mathrm{vert}}

\newcommand{\VW}{\mathrm{VW}}

\newcommand{\vectorL}{L}

\newtheorem{result}{Theorem}

\numberwithin{equation}{section}
\newtheorem{theorem}[equation]{Theorem}

\newtheorem{corollary}[equation]{Corollary}
\newtheorem{definition}[equation]{Definition}
\newtheorem{lemma}[equation]{Lemma}
\newtheorem{proposition}[equation]{Proposition}
\newtheorem{conjecture-definition}[equation]{Conjecture-Definition}

\theoremstyle{remark}

\newtheorem{remark}[equation]{Remark}

\makeatletter
\def\l@subsection{\@tocline{2}{0pt}{2pc}{6pc}{}}
\makeatother

\begin{document}

\begin{abstract}
We show that \emph{vertical} contributions to (possibly semistable) Tanaka-Thomas-Vafa-Witten invariants are well defined for surfaces with $p_g(S)>0$, partially proving conjectures of \cite{TT2} and \cite{T}. Moreover, we show that such contributions are computed by the same tautological integrals as in the stable case, which we studied in \cite{L}. Using the work of Kiem and Li, we show that stability of universal families of vertical Joyce-Song pairs is controlled by cosections of the obstruction sheaves of such families.
\end{abstract}

\maketitle

\tableofcontents

\section{Introduction}
\subsection{Joyce-Song pairs}\label{SecModuliSpace}
We will consider Joyce-Song pairs and their moduli spaces, which have been studied in \cite{JS} and (in our context) in \cite{TT2}. Let \(S\) be an smooth algebraic surface over \(\CC\) with a polarisation \(H\), and let
\[q\colon X \rightarrow S\]
be the total space of the canonical bundle of \(S\). We will consider moduli spaces of certain compactly supported coherent sheaves on $X$. We are particularly interested in strictly semistable sheaves. Since semistable sheaves may have non-trivial automorphisms, we will work with sheaves that have been rigidified by a Joyce-Song section.

\begin{definition}
Let $m \gg 0$ be an integer. A \emph{Joyce-Song pair} is a pair $(\E, s)$ consisting of a coherent compactly supported sheaf $\E$ on $X$, and a non-zero section
\[s\colon \O_X(-m) \rightarrow \E\,.\]
A Joyce-Song pair $(\E,s)$ is called \emph{stable} if $\E$ is Gieseker semistable, and $s$ does not factor through any strict subsheaf $\F\subsetneq\E$ for which we have an equality
\[p(\F) = p(\E)\]
of reduced Hilbert polynomials.
\end{definition}

We fix a \emph{charge} $\gamma = (r,c_1,c_2) \in H^\mathrm{even}(S)$ with \(r \geq 1\), and a line bundle $L$ on $S$ with $c_1(L) = c_1$. We will write $\gamma^\perp = (r,L,c_2)$.
Let $\E$ be a compactly supported coherent sheaf on $X$. We will say that $\E$ is of \emph{type $\gamma^\perp$} if
\begin{itemize}
\item $\rk(q_*\E) = r$
\item $\det(q_*\E) \cong L $
\item $c_2(q_*\E) = c_2$
\item $\E$ has zero {\it centre of mass} (see \cite{TT})\,.
\end{itemize}

We will write
\[ \P^\perp = \P^\perp_{\gamma}(m) = \Big\{ \text{stable Joyce-Song pairs }(\E,s) \text{ with $\E$ of type $\gamma^\perp$} \Big\}\]
for the moduli space of Joyce-Song pairs.

\subsection{Vafa-Witten invariants}
A point \( (\E,s)\in\P^\perp\) can be viewed as an object
\[I^\bullet = [\O_X(-m) \xrightarrow{s} \E] \in \mathcal{D}^b(X)\]
in the derived category of $X$, where \(\O_X(-m)\) is placed in degree 0. By \cite{TT2, JS}, the moduli space \(\P^\perp\) carries a perfect obstruction theory governed by
\[R\Hom_X(I^\bullet,I^\bullet)_\perp\,,\]
which is given by
\[R\Hom_X(I^\bullet,I^\bullet) \cong R\Hom_X(I^\bullet,I^\bullet)_\perp 
	\oplus H^*(\O_X) \oplus H^{\geq1}(\O_S) \oplus H^{\leq1}(K_S)[-1]\,.\]
Vafa-Witten invariants have been defined \emph{conjecturally} in \cite{TT2,T}. We will give their definition for an algebraic surface $S$ with $p_g(S)>0$.

Consider the natural \(\CC^*\) action on \(X\rightarrow S\), given by scaling the fibres. It induces a \(\CC^*\) action on \(\P^\perp\). The unrefined invariant will be defined by the virtual localization formula \cite{GP}.

\begin{conjecture-definition}\textup{\cite{TT2}} \label{ConDefUnref}
There exists a rational number $\VW_{\gamma}$, called the \emph{Vafa-Witten invariant} of \((S,H, \gamma^\perp)\), such that for all $m\gg 0$, we have
\begin{align*}
\int_{\left[\P^\perp_\gamma(m)\right]^\vir} 1
	\coloneqq{} & \int_{\left[(\P^\perp_\gamma(m))^{\CC^*}\right]^\vir} \frac{1}{e(N^\vir)} \\
	={} & (-1)^{\chi(\gamma(m))-1} \chi(\gamma(m)) \VW_{\gamma}\,.
\end{align*}
\end{conjecture-definition}

Refined invariants have been defined in \cite{T} by the K-theoretic virtual localisation formula \cite{Qu,CK}. For notation and definitions, see \cite{T}.

\begin{conjecture-definition}\textup{\cite{T}}\label{ConDefRef}
There exists a rational function $\VW_{\gamma}(t)$ in $\sqrt{t}$, called the \emph{refined Vafa-Witten invariant} of \((S,H, \gamma^\perp)\), such that for all $m\gg 0$, we have
\begin{align*}
\chi_t\left(\P^\perp_\gamma(m),
\hat\O^\vir_{\P^\perp_\gamma(m)} \right)
	\coloneqq{} & \chi_t\left(\left(\P^\perp_\gamma(m)\right)^{\CC^*},
	\frac{\O^\vir_{\left(\P^\perp_\gamma(m)\right)^{\CC^*}}}
	{\Lambda^\bullet(N^\vir)^\vee}
	\otimes	\left.K^\frac{1}{2}_{\P^\perp_\gamma(m), \vir}\right|_{\left(\P^\perp_\gamma(m)\right)^{\CC^*}} \right) \\
	={} & (-1)^{\chi(\gamma(m))-1} [\chi(\gamma(m))]_t \VW_{\gamma}(t)\,,
\end{align*}
where
\[
[\chi(\gamma(m))]_t \coloneqq \frac{ t^{\chi(\gamma(m))-1} + \ldots + t + 1}{t^\frac{\chi(\gamma(m))-1}{2}}
\]
denotes the \emph{quantum integer}.
\end{conjecture-definition}

\subsection{The fixed locus} \label{SubSecFixed}
Following \cite{GSY2}, the fixed locus \((\P^\perp)^{\CC^*}\) can be written as a disjoint union of open and closed subschemes as follows. Let \((\E,s)\in (\P^\perp)^{\CC^*}\) be a \(\CC^*\) fixed Joyce-Song pair. By Lemma~\ref{LemmaCanonES} below, we have a canonical weight space decomposition
\[E = q_* \E = E^0 \oplus \ldots \oplus E^{-k}\,.\]
We will write
\[\lambda_{(\E,s)} \coloneqq (\rk E^0, \ldots, \rk E^{-k})\,\]
for the vector of ranks. We now have, cf.\ loc.\ cit., a decomposition
\[(\P^\perp)^{\CC^*} = \coprod_{|\lambda| = r} \P^\perp_\lambda\,,\]
where \(\P_\lambda^\perp\) is the open and closed subscheme defined by
\[\lambda_{(\E,s)} = \lambda\]
for an ordered partition \(\lambda = (\lambda_0, \ldots, \lambda_k) \) of \(r\). We will study the contribution of the locus \(\P^\perp_{1^r} = \P^\perp_{(1,\ldots,1)}\) of \emph{vertical} Joyce-Song pairs to the Vafa-Witten invariants.

Using \cite{GSY} and \cite{GT2}, we will construct the schemes \(\P^\perp_{1^r}\) and their virtual classes directly in Sections~\ref{SecConstruction}-\ref{SecLocal}.

\subsection{Results}\label{SubsectionResults}
Similar to Conjecture-Definition~\ref{ConDefUnref}, we can define \emph{vertical} contributions to the Vafa-Witten invariant by
\[\int_{\left[\P^\perp_{1^r}\right]^\vir} \frac{1}{e(N^\vir)} = (-1)^{\chi(\gamma(m))-1} \chi(\gamma(m)) \VW_{\gamma}^\vertical\]
for \(m\gg0\) (note that the left hand side depends implicitly on \(m\)).
The vertical contribution \(\VW_\gamma^\vertical(t)\) to the refined invariant can be defined by
\begin{align*}
\chi_t\left(\P^\perp_{1^r},
	\frac{\O^\vir_{\P^\perp_{1^r}}}
	{\Lambda^\bullet(N^\vir)^\vee}
	\otimes	\left.K^\frac{1}{2}_{\P_\gamma^\perp(m), \vir}\right|_{\P^\perp_{1^r}} \right)
	= (-1)^{\chi(\gamma(m))-1} [\chi(\gamma(m))]_t \VW_{\gamma}^\vertical(t)
\end{align*}
for \(m\gg0\).

\begin{result}\label{ResultA}
Let $S$ be a surface with $p_g(S)>0$, with a polarisation \(H\) and \(\gamma^\perp = (r,L,c_2)\) given as above. The vertical contributions
\[\VW_{\gamma}^\vertical \quad\text{and}\quad \VW_{\gamma}^\vertical(t)\]
to the (refined) Vafa-Witten invariant of \((S,H,\gamma^\perp)\) are well-defined.
\end{result}
%

Let $(S,H)$ be a polarised surface with $H^1(\O_S) = 0$ and $p_g(S)>0$. Let $\gamma$ be a charge for which any semistable sheaf on $X$ of type $\gamma$ is stable. In \cite{L}, we have seen that the vertical contributions to the Vafa-Witten invariant of \((S,H,\gamma)\) can be expressed in terms of the coefficients of universal Laurent series \(A, B, C_{ij} \in \QQ(\!(q^\frac{1}{2r})\!)\) for \(1\leq i\leq j < r\) and Seiberg-Witten invariants \(\SW(\beta^i)\) of classes \(\beta^i \in H^2(S,\ZZ)\). The Laurent series are in turn defined by certain tautological integrals over products of Hilbert schemes of points on \(S\).

We will show that the tautological integrals compute the Vafa-Witten invariants for any surface $S$ with $p_g(S)>0$ and any $\gamma^\perp$. More precisely, consider the following generating series
\[\Z_{S,r,L}(q) = \frac{q^{\frac{1-r}{2r} c_1(L)^2}}{\#\Pic(S)[r] } \, \sum_{c_2\in\ZZ} \VW_{(r,L,c_2)}^\vertical\, q^{c_2}\,,\]
where \(\#\Pic(S)[r]\) is the $r$-torsion of the Picard group of \(S\). Then the series \(A,B,C_{ij}\) in the following theorem, or rather their refined counterparts, are precisely the ones of \cite[Theorem A]{L}.
\begin{result}\label{ResultB}
Fix a rank \(r\geq 1\). There exist universal Laurent series
\[A,B,C_{ij} \in \QQ(\!(q^\frac{1}{2r})\!) \,, \quad 1\leq i \leq j < r\,,\]
depending only on \(r\), such that for any surface \(S\) with \(p_g(S)>0\), and any line bundle $L$ on $S$, we have
\[
\Z_{S,r,L}(q) = 
	A^{\chi(\O_S)}
	B^{K_S^2} \,
\sum_\beta
	 \SW(\beta^1)\cdots\SW(\beta^{r-1})
	\prod_{i\leq j}C_{ij}^{\beta^i\beta^j}
\]
where the sum is taken over classes \(\beta^1,\ldots,\beta^{r-1}\in H^2(S,\ZZ)\) with
\[c_1(L) \equiv \sum_i i\beta^i \mod rH^2(S,\ZZ)\,.\]
The same statement holds for generating series of vertical contributions to refined Vafa-Witten invariants, when one allows the Laurent series to have coefficients in \(\QQ(\sqrt t)\).
\end{result}
\begin{remark}
In particular \(\VW_\gamma^\vertical\) and \(\VW_\gamma^\vertical(t)\) do not depend on the polarisation \(H\), or on the lift \(\gamma^\perp\) of \(\gamma\).
\end{remark}

\subsection{Acknowledgement}
I thank Amin Gholampour, my Ph.D. advisor Martijn Kool, and Richard Thomas for useful discussions. I thank Richard Thomas for hosting me at Imperial College, where part of the work presented here was done.

\section{A tautological family of Joyce-Song pairs} \label{SecConstruction}
Recall that via the spectral construction (see e.g.\ \cite{TT}), a sheaf $\E$ on $X$ can be viewed as Higgs pair \((E,\phi)\) on \(S\), where \(E = q_* \E\) is a sheaf on \(S\), and \(\phi\colon E\rightarrow E\otimes \omega_S\) is a map that encodes the \(\O_X\)-module structure of \(\E\). In \cite{L}, we have studied families of Higgs pairs that are flags of sheaves of rank one. Such families form (\'etale covers of) nested Hilbert schemes. We will equip the Higgs pairs, or their corresponding sheaves on $X$, with a Joyce-Song section. The resulting Joyce-Song pairs will form a projective bundle over the space of Higgs pairs.

Choose classes $\alpha_0,\ldots, \alpha_s \in H^2(S,\ZZ)$ and a line bundle \(L\) on \(S\) with
\[c_1(L) = \alpha_0+\ldots+\alpha_s\,.\]
Define classes
\begin{align*}
\beta_1
	& = \alpha_1-\alpha_0 + c_1(\omega_S) \\
	& \vdotswithin{=} \\
\beta_{s}
	& = \alpha_{s} - \alpha_{s-1} + c_1(\omega_S) \,,
\end{align*}
and write \(S_{\beta_i}\) for the Hilbert scheme of curves on \(S\) with class \(\beta_i\).
Let $M_\alpha$ be the limit of the following (solid) diagram:
\[
\begin{tikzcd}
M_\alpha \arrow[dd,dotted] \arrow[rr, dotted] \arrow[dr, dotted]& & S_{\beta_1} \times \cdots \times S_{\beta_s} \arrow[d] \\
& \Pic_{\alpha_0}(S) \times \cdots \times \Pic_{\alpha_s}(S) \arrow[r, "{\partial}"] \arrow[d,"{\det}"] & \Pic_{\beta_1}(S) \times \cdots \times \Pic_{\beta_{s}}(S) \\
{[L]} \arrow[r] & \Pic_{\alpha_0 + \ldots + \alpha_s}(S) &
\end{tikzcd}
\]
where the map \(\det\) is given by the rule
\[(L_0,\ldots,L_s) \mapsto L_0\otimes\cdots \otimes L_s\,,\]
and the map \(\partial\) by
\[(L_0,\ldots,L_s) \mapsto (L_0^* \otimes L_1 \otimes \omega_S, \ldots, L_{s-1}^* \otimes L_s \otimes \omega_S)\,.\]
Then \(M_\alpha\) parametrizes sums of line bundles
\[L_0\oplus \ldots \oplus L_s\,,\]
with constant determinant \(L\), together with non-zero maps
\begin{equation}\label{EqMapsPhi}
\phi_i \colon  L_{i-1}
	 \rightarrow L_i \otimes \omega_S \quad \text{for } i = 1,\ldots,s \,.
\end{equation}

For non-negative integers \(n_0,\ldots,n_s\), let \(S^{[n_i]}\) be the Hilbert scheme of \(n_i\) points on \(S\). We will write
\[S_\beta^{[n]} = S_{\beta_1,\ldots,\beta_s}^{[n_0,\ldots,n_s]} \hookrightarrow S^{[n_0]}\times \cdots \times S^{[n_s]} \times  S_{\beta_1}\times \cdots \times S_{\beta_{s}}\]
for the nested Hilbert scheme, i.e.\ the subscheme defined by the rule
\[I_{i-1}(-C_i) \subset I_i \,, \quad i = 0,\ldots,s\]
for ideal sheaves \(I_i \in S^{[n_i]}\) and curves \(C_i \in S_{\beta_i}\).
Let \(M_\alpha^n\) be the fibre product
\[
\begin{tikzcd}
M_\alpha^n \arrow[d,dotted] \arrow[r,dotted] & S_\beta^{[n]} \arrow [d] \\
M_\alpha \arrow[r] & S_{\beta_1}\times \cdots \times S_{\beta_{s}} \,.
\end{tikzcd}
\]
Then \(M_\alpha^n\) parametrizes Higgs pairs \((E,\phi)\) on \(S\), given by a sheaf \(E\) with \(\det (E) \cong L\) of the form
\[E = (L_0\otimes I_0) \oplus \ldots \oplus (L_s \otimes I_s)\,,\]
with line bundles \(L_i \in \Pic_{\alpha_i}(S)\) and ideal sheaves \(I_i \in S^{[n_i]}\) for \(i = 0,\ldots,s\), together with maps as in \eqref{EqMapsPhi}, that factor through the ideal sheaves:
\[
\begin{tikzcd}
L_0 \arrow[r, "\phi_1"] & L_1 \otimes \omega_S \arrow[r, "\phi_2"] & \cdots \arrow[r, "\phi_s"] & L_s \otimes \omega_S^{\otimes s} \\
L_0\otimes I_0 \arrow[r, dotted, "\phi_1"] \arrow[u] & L_1\otimes I_1 \otimes \omega_S \arrow[r, dotted, "\phi_2"] \arrow[u] & \cdots \arrow[r, dotted, "\phi_s"] & L_s\otimes I_s \otimes \omega_S^{\otimes s} \arrow[u] \,.
\end{tikzcd}
\]

\begin{remark}\label{RemTorsionFree}
A torsion-free sheaf of rank one on \(S\) can be uniquely written as \(I\otimes L\), with \(I\) an ideal sheaf of a finite subscheme of \(S\), and \(L\) a line bundle. Writing
\[E^{-i} = L_i\otimes I_i \,, \quad i=0,\ldots,s\,,\]
it follows that we can view \(M_\alpha^n\) is a moduli space of graded sums of rank one torsion-free sheaves
\[E = E^0 \oplus \ldots \oplus E^{-s}\]
on \(S\), with a homogeneous Higgs field \(\phi\colon E \rightarrow E\otimes \omega_S\) of weight \(-1\) and rank \(s\).
\end{remark}

We choose \emph{a} universal Higgs pair \((E,\phi)\) on \(M_\alpha^{n} \times S\), which is only unique up to twists by elements of \(\Pic (M_\alpha^{n})\). Let \(\E\) be the sheaf on \(M_\alpha^{n} \times X\) corresponding to \((E,\phi)\).

Fix an integer \(m\gg0\) and write
\[E_m = E \otimes \O_S(m\cdot H)
\quad\text{and}\quad
\E_m = \E \otimes \O_X(m \cdot q^*H)\,.
\]
Let
\[\pi_S\colon M_\alpha^n \times S \rightarrow M_\alpha^n \,, \quad \pi_X\colon M_\alpha^n \times X \rightarrow M_\alpha^n\]
denote the projections. Define a projective bundle
\begin{align*}
\PP \coloneqq{}
	&  \PP(\pi_{X*}\E_m)\\
={}
	& \PP(\pi_{S*}E_m) \,.
\end{align*}
with projection map
\[p\colon \PP \rightarrow M_\alpha^{n}\]
and canonical line bundle \(\O_\PP(1)\). The tautological section
\[s\colon \O_{\PP\times X} \rightarrow (p\times \mathrm{id}_X)^*\E_m \otimes \O_\PP(1)\,,\]
defines a universal family of Joyce-Song pairs on \(\PP\times X\), which we will denote by
\begin{equation}\label{EqUnivJS}
(\E(1), s)\,.
\end{equation}

\section{The virtual class}\label{SectionVirtualClass}
Note that \(M_\alpha \rightarrow S_{\beta_1}\times \cdots \times S_{\beta_{s}} \) is an surjective \'etale morphism of degree \((s+1)^{2\cdot q(S)}\), with
\[q(S) = h^{0,1}(S) = \dim H^1(S,\O_S)\]
the irregularity of \(S\). In fact, it is a torsor under the torsion subgroup
\[(\ZZ/ (s+1) \ZZ )^{2 \cdot q(S)} \cong\Pic_0(S)[s+1] \subset \Pic_0(S)\,,\]
which acts on \(M_\alpha\) by the rule
\[N \cdot (E, \phi) = (E \otimes N, \phi)\]
for \(N\in \Pic_0(S)\) with with \(N^{\otimes(s+1)} \cong \O_S\) and \((E, \phi) \in M_\alpha\). The same holds for
\[\eta \colon M_\alpha^n \rightarrow S_\beta^{[n]} \,.\]
In particular \(M_\alpha^n\) carries a perfect obstruction theory, which is simply the pull-back of the perfect obstruction theory on \(S_\beta^{[n]}\) considered in \cite{GSY} and \cite{GT2}. Note that we have
\begin{align*}
\eta^* [S_\beta^{[n]}]^\vir
	& = [M_\alpha^n]^\vir \quad \text{and}\\
\eta_* [ M_\alpha^n]^\vir 
	& = (s+1)^{2\cdot q(S)} \cdot [S_\beta^{[n]}]^\vir \,.
\end{align*}

Let \(T_{M_\alpha^n}\) denote the virtual tangent bundle of \(M_\alpha^n\), i.e.\ the class in \(K^0(M_\alpha^n)\) of the dual of its perfect obstruction theory.

\begin{proposition}\label{PropVirPre}
The scheme \(\PP\) carries a perfect obstruction theory with virtual tangent bundle
\[T_\PP = p^*T_{M_\alpha^n} + T_{\PP/M_\alpha^n} - T_{\PP/M_\alpha^n}^\vee\,.\]
\end{proposition}
\begin{proof}
We will give the proof for rank two. The general case follows directly from the techniques of \cite[Section 5]{GT2}. In Section 4.3 of loc.cit., a vector bundle
\[B \rightarrow \Pic_{\beta}(S) \times S^{[n_0]} \times S^{[n_1]}\,\]
is constructed, together with an open subscheme \(U\subset \PP(B)\) and a vector bundle \(F\) on \(U\), such that \(S_\beta^{[n_0,n_1]} \subset U\)
is the zero locus of a section
\[ s\colon \O_U \rightarrow F(1)\,, \]
where we have twisted \(F\) by (the restriction of) the canonical line bundle on \(\PP(B)\). The perfect obstruction theory of \(S_\beta^{[n_0,n_1]}\) is now given by
\[\left[ (F(1))^\vee|_{S_\beta^{[n_0,n_1]}} \rightarrow \Omega_U |_{S_\beta^{[n_0,n_1]}}\right] \in \D^b(S_\beta^{[n_0,n_1]})\,.\]
I claim that there exist a cartesian square
\[
\begin{tikzcd}
\PP \arrow[d]\arrow[r]
	& \overline{\PP} \arrow[d] \\
S_\beta^{[n_0,n_1]} \arrow[r]
	& \PP(B)\,,
\end{tikzcd}
\]
where the vertical maps are smooth morphims. It follows that \(\PP\) is the zero locus in \(U_{\overline{\PP}}  = \overline{\PP} \times_{\PP(B)} U \subset \overline{\PP}\) of the section
\[(\pr^*s,0) \colon \O_{U_{\overline{\PP}}} \rightarrow \pr^*F(1) \oplus T_{\overline{\PP}/{\PP(B)}}^\vee\,,\]
where \(\pr\colon U_{\overline{\PP}} \rightarrow U\) denotes the projection. The complex
\[\left[ ((\pr^*F(1))^\vee \oplus T_{\overline{\PP}/{\PP(B)}})|_{\PP} \rightarrow \Omega_{U_{\overline{\PP}}} |_{\PP}\right] \in \D^b(\PP)\,\]
defines a perfect obstruction theory on \(\PP\), which satisfies the description of the proposition. In fact, we have
\begin{align*}
T_\PP
	& = \left(T_{U_{\overline{\PP}}} - F(1) - T_{\overline{\PP}/\PP(B)}^\vee\middle)\right|_\PP \\
	& = \left(\pr^*T_{U} + T_{\overline{\PP}/\PP(B)} - F(1) - T_{\overline{\PP}/\PP(B)}^\vee\middle)\right|_\PP \\
	& = p^*\eta^* (T_U - F(1))|_{S_\beta^{[n_0,n_1]}} + T_{\PP/S_\beta^{[n_0,n_1]}} - T_{\PP/S_\beta^{[n_0,n_1]}}^\vee \\
	& = p^*T_{M_\alpha^n} + T_{\PP/M_\alpha^n} - T_{\PP/M_\alpha^n}^\vee \,.
\end{align*}
We will now prove the claim.


Let \(\iota\), \(\sigma\) and \(\tau\) be the morphisms in the diagram
\[
\begin{tikzcd}
S_\beta^{[n_0,n_1]} \arrow[dd, swap, "\sigma"] \arrow[r, "\iota"] \arrow[ddr, bend right=20, swap, "\tau"]
	&\PP(B) \arrow[d]
\\
	& \Pic_{\beta}(S) \times S^{[n_0]} \times S^{[n_1]} \arrow[d]
\\
S_\beta \arrow[r]
	& \Pic_{\beta}(S) \,,
\end{tikzcd}
\]
and let \( \D_\beta \subset S_\beta \times S \) be the universal divisor. The bundle \(B\) depends on the choice of a Poincar\'e bundle \(\L_{\beta_i}\) on \(\Pic_\beta(S) \times S\), and by \cite[Section 4.2 - Equation (4.21)]{GT2} we have an isomorphism
\begin{equation}\label{EqRestricsToCanon}
(\tau \times \mathrm{id}_S)^* \L_{\beta_i} \otimes (\iota\times\mathrm{id}_S) ^*\O_{\PP(B)}(1) \cong (\sigma \times \mathrm{id}_S)^* \O(\D_\beta)\,
\end{equation}
of line bundles on \(S_\beta^{[n_0,n_1]} \times S\). Let \(Y\) be scheme completing the following cartesian diagram with \'etale vertical morphisms
\[
\begin{tikzcd}
M_\alpha^n \arrow[d, "\eta"] \arrow[r]
	& Y \arrow[d, "\eta"] \arrow[r]
	& (\Pic_{\alpha_0}(S) \times \Pic_{\alpha_1}(S))_{[L]}\arrow[d, "\eta"]
\\
S_\beta^{[n_0,n_1]} \arrow[r]
	&\PP(B) \arrow[r]
	& \Pic_{\beta}(S) \,.
\end{tikzcd}
\]
Let \(\L_{\alpha_0}\) be the Poincar\'e bundle on \(\Pic_{\alpha_0}(S) \times S\), and for \(i=0,1\), let \(\I^{[n_i]}\) denote the universal ideal sheaf on \(S^{[n_i]}\times S\). Define a sheaf
\[\overline{E} = \Big(\L_{\alpha_0} \otimes \I^{[n_0]}\Big) \oplus \Big(\L_{\alpha_0} \otimes \omega_S^* \otimes \L_{\beta}\otimes \O_{\PP(B)}(1) \otimes \I^{[n_1]}\Big) \]
on \(Y\times S\), where we have suppressed the various pull-backs. By \eqref{EqRestricsToCanon}, its restriction to \(M_\alpha^n\times S\) equals the sheaf \(E\) defined in Section \ref{SecConstruction}  (up to a twist by an element of \(\Pic(M_\alpha^n)\)) . Hence, possibly after increasing \(m\), we can define
\[\overline{\PP} \coloneqq \PP\bigg(\pi_{S*}\Big(\overline{E}\otimes \O_S(m\cdot H)\Big)\bigg) \rightarrow Y\,,\]
proving the claim.
\end{proof}
Proposition \ref{PropVirPre}, or more precisely the claim in its proof, has the following consequences.
\begin{corollary}\label{CorVirClass}
The scheme \(\PP\) carries a natural perfect obstruction theory with virtual cycle
\[ [\PP]^\vir = p^*[M_\alpha^n]^\vir \cap e(T_{\PP/M_\alpha^n}^\vee) \,.\]
\end{corollary}
\begin{corollary}\label{CorVirStruc}
The scheme \(\PP\) carries a virtual structure sheaf, which is given by
\[ \O_{\PP}^\vir = p^*\O_{M_\alpha^n}^\vir \otimes \Lambda^\bullet \left(T_{\PP/M_\alpha^n}\right) \,.\]
\end{corollary}

\section{Localising the virtual class} \label{SecLocal}
The family of sheaves on \(\PP\times X\)
\[\E(1) = (p \times \mathrm{id}_X)^* \E\otimes \O_\PP(1)\]
constructed in Section \ref{SecConstruction} is invariant under the natural \(\CC^*\) action on the fibres of \(X\rightarrow S\), and can be equipped with an equivariant structure. In terms of the Higgs pair \((E,\phi)\) on \(M_\alpha^n \times S\), it is given by a \(\CC^*\) action on \(E\), which is chosen in such a way that the Higgs field \(\phi\) acts with weight \(-1\) on \(E\). We may assume that the highest weight occurring in the weight space decomposition of \(E\) is zero. Following Remark~\ref{RemTorsionFree}, we write
\[E = E^0 \oplus \ldots \oplus E^{-s}\,,\]
and let \(\CC^*\) act on \(E^{-i}\) with weight \(-i\).

The action induces an action on the family of Joyce-Song pairs \(\PP\). In fact, we have
\[\PP= \PP\left(\pi_{S*}E^0_m \oplus \ldots \oplus \pi_{S*}E^{-s}_m\right)\]
with
\[E^{-i}_m = E^{-i} \otimes \O_S(mH)\,, \quad i = 0,\ldots,s \,,\]
so the \(\CC^*\) action on \(\PP\) is given by
\[\lambda \cdot (x_0:\ldots:x_s) \rightarrow (x_0 \lambda^0:\ldots:x_s \lambda^{-s})\]
for \(\lambda \in \CC^*\) and \((x_0:\ldots:x_s) \in \PP\). For \(i = 0,\ldots, s\), we will write
\[\PP_i \coloneqq \PP\left(\pi_{S*}E^{-i}_m\right) \xrightarrow{p_i} M_\alpha^n\,\]
for the projective bundle, so the \(\CC^*\) fixed locus of \(\PP\) is given by
\[\PP_0\sqcup\ldots\sqcup \PP_s = \PP^{\CC^*} \subset \PP\,.
\]

Let \(\t\) be an equivariant parameter for the \(\CC^*\) action on a point. We can equip the perfect obstruction theory of Proposition~\ref{PropVirPre} with an equivariant structure, so that its virtual tangent bundle is given by
\begin{equation}\label{EqEquivStrucTan}
T_\PP = p^*T_{M_\alpha^n} + T_{\PP/M_\alpha^n} - T_{\PP/M_\alpha^n}^\vee \otimes \t
\end{equation}
in \(K^0_{\CC^*}(\PP)\).
On each $\PP_i$ with $i> 0$, we have an induced \(\CC^*\) localized perfect obstruction theory with virtual tangent bundle
\begin{align}\label{EqObstructionPi}
\left(T_\PP\vert_{\PP_i}\right)^\text{fix}
	& =p_i^*T_{M_\alpha^n} + \left( \left. \left(T_{\PP/M_\alpha^n} - T_{\PP/M_\alpha^n}^\vee \otimes \t \right) \right\vert_{\PP_i}\right)^\text{fix} \\
	& = p_i^*T_{M_\alpha^n} +T_{\PP_i/M_\alpha^n} - \left(p_i^*\pi_{S*}E_m^{-(i-1)}(1)\right)^\vee \notag
\end{align}
and virtual class
\begin{equation}\label{EqCapEul}
[\PP_i]^\vir = p_i^*[M_\alpha^n]^\vir \cap e\left(\left(p_i^*\pi_{S*}E_m^{-(i-1)}(1)\right)^\vee\right)\,.
\end{equation}
The \(\CC^*\) localized virtual tangent bundle on $\PP_0$ is given by
\[
\left(T_\PP\vert_{\PP_0}\right)^\text{fix}
= p_0^*T_{M_\alpha^n} + T_{\PP_0/M_\alpha^n}\,,\]
and we have
\begin{equation}\label{EqCapGeenEul}
[\PP_0]^\vir = p_0^*[M_\alpha^n]^\vir\,.
\end{equation}
Similarly we have
\[
\O_{\PP_i}^\vir =
\begin{cases}
p_0^* \O^\vir_{M_\alpha^n} & \text{if } i = 0 \\
p_i^* \O^\vir_{M_\alpha^n} \otimes \Lambda^\bullet \left(p_i^*\pi_{S*}E_m^{i-1}(1)\right) & \text{if }  i> 0 \,.
\end{cases}
\]

\section{Stability}
Let \((E,\phi)\) be the family of Higgs pairs on
\[\pi_S\colon M_\alpha^n \times S \rightarrow M_\alpha^n\]
constructed in Section~\ref{SecConstruction}. Recall that \(\phi\) is given by maps
\[\phi_i\colon E^{-(i-1)} \rightarrow E^{-i}\otimes\omega_S \,, \quad i = 1,\ldots,s\,,\]
where the torsion free sheaves of rank one \(E^{-i}\) are the summands of
\[E = E^0 \oplus \ldots \oplus E^{-s}\,.\]
We have defined the perfect obstruction theory of \(M_\alpha^n\) as the pull-back along the \'etale map \(\eta\colon M_\alpha^n \rightarrow S_\beta^{[n]}\) of the perfect obstruction theory on the nested Hilbert scheme \(S_\beta^n\). It follows directly by the description of \cite{GSY} that it is given by given by the dual of a cone on
\begin{equation}\label{EqPOTM}
\left(\bigoplus_{i=0}^s \RHom_{\pi_S} (E^{-i},E^{-i}) \right)_0 \xrightarrow{ \circ \phi - \phi \circ} \bigoplus_{i=1}^s \RHom_{\pi_S}(E^{-(i-1)}, E^{-i}\otimes \omega_S)\,,
\end{equation}
where the left hands side is given by the kernel of the trace map
\[\bigoplus_{i=0}^s \RHom_{\pi_S} (E^{-i},E^{-i}) \hookrightarrow \RHom_{\pi_S} (E,E) \xrightarrow{\mathrm{tr}} R\pi_{S*}\O_S.\]

\begin{proposition}
For each \(i = 1,\ldots,s\), we have an exact sequence
\begin{equation}\label{EqCosections}
\begin{tikzcd}[column sep = small, row sep = small]
ob_{M_\alpha^n} \arrow[r, "\sigma_i"]
	& R^2\pi_{S*} \O_S \arrow[r, "\phi_i"]
	& \EExt_{\pi_S}^2(E^{-(i-1)}, E^{-i}\otimes \omega_S) \arrow[r]
	& 0 \,.
\end{tikzcd}
\end{equation}
\end{proposition}
\begin{proof}
By \eqref{EqPOTM} have an exact sequence
\[
\begin{tikzcd}[column sep = small, row sep = small]
ob_{M_\alpha^n} \arrow[r]
	& \left(\bigoplus_{i=0}^s \EExt_{\pi_S}^2 (E^{-i},E^{-i}) \right)_0 \arrow[r] \arrow[d, "\cong"]
	& \bigoplus_{i=1}^s \EExt_{\pi_S}^2(E^{-(i-1)}, E^{-i}\otimes \omega_S) \arrow[r]
	& 0 \,.\\
	& \bigoplus_{i=1}^s R^2\pi_{S*} \O_S \arrow[ru, dotted, swap, "\phi"]
	&
	&
\end{tikzcd}
\]
Let the vertical isomorphism is given by
\[(a_0,\ldots,a_s) \mapsto (\tr(a_1) - \tr(a_0),\ldots, \tr(a_s) - \tr(a_{s-1}))\,,\]
so the dotted arrow \(\phi\) is the map given by the rule
\[(b_1,\ldots,b_s) \mapsto (b_1 \cdot \phi_1,\ldots, b_s \cdot \phi_s)\,.\]
It follows that for each summand \(\phi_i\) of \(\phi\), we have an exact sequence \eqref{EqCosections}.
\end{proof}

\begin{remark}
In particular, we see that if \(p_g(S)>0\), and if
\(\alpha_{i-1} - \alpha_i\)
is not the class of an effective divisor (so \(\EExt_{\pi_S}^2(E^{-(i-1)}, E^{-i}\otimes \omega_S)=0\) by Serre duality), the obstruction sheaf \(ob_{M_\alpha^n}\) has a trivial factor \(H^2(\O_S)\). In rank \(s+1 = 2\), this fact is used in \cite[Section 2.2]{GSY} to define a reduced virtual class on \(S_\beta^{[n_0,n_1]}\).
\end{remark}

Recall that we write \(\E\) for the sheaf on \(M_\alpha^n\times S\), corresponding to \((E,\phi)\) via the spectral construction.

\begin{proposition}\label{PropCases}
Assume \(p_g(S)>0\). Then (at least) one of the following statements holds:
\begin{enumerate}[i)]
\item The family \(\E\) of sheaves on \(M_\alpha^n\times X\rightarrow M_\alpha^n\) is fibrewise Gieseker stable;
\item We have \(\alpha_0 = \ldots = \alpha_s\) and \(n_0 = \ldots = n_s\) (so in particular \(\E\) is fibrewise strictly Gieseker semistable);
\item The virtual class \([M_\alpha^n]^\vir\) vanishes.
\end{enumerate}
\end{proposition}
\begin{remark}
In \cite[Section 2.2]{GSY} (see the discussion above) and in \cite[Section 4.1]{GT2}, it is shown that \([S_\beta^{[n]}]^\vir\) vanishes, unless the classes
\[\alpha_{i-1} - \alpha_{i}\,, \quad i = 1,\ldots,s\]
are classes of effective divisors. By an easy lemma \cite[Lemma 3.2]{L}, \(\E\) is slope semistable in the latter case.
\end{remark}
\begin{remark}
In \cite[Proposition 3.5]{L}, we have proven a similar statement for the class
\[\iota_*[S_\beta^{[n]}]^\vir \in A^*(S^{[n_0]} \times \cdots \times S^{[n_s]} \times S_{\alpha_1} \times \cdots \times S_{\alpha_s})\]
under the condition \(H^1(S,\O_S) = 0\), using the formula of \cite[Theorem 5.6]{GT2}.
\end{remark}
\begin{proof}

Assume that we are not in case i) or ii). By Lemmas~3.2 and 3.4 of \cite{L}, there is an \(i\in \{1,\ldots,s\}\) such that
\begin{itemize}
\item \(\alpha_{i-1} - \alpha_i\) is not the class of an effective divisor, or
\item \(\alpha_{i-1} = \alpha_i\) and \(n_{i-1} > n_i\).
\end{itemize}
In either case, we have
\[\Ext_{\O_S}^2(E^{-(i-1)}_x, E^{-i}_x\otimes\omega_S)
	= \Hom_{\O_S}(E^{-i}_x, E^{-(i-1)}_x)^* = 0\,, \quad
	\forall x\in M_\alpha^n(\CC)\,.\]
By the exact sequence \eqref{EqCosections}, the map
\[ob_{M_\alpha^n} \xrightarrow{\sigma_i} R^2\pi_{S*} \O_S \neq 0\]
is surjective. It follows that \([M_\alpha^n]^\vir = 0\) by e.g.\ \cite{KL}.
\end{proof}

For the universal family of Joyce-Song pairs \((\E(1),s)\) on
\[\pi_X\colon\PP\times X \rightarrow \PP\]
defined in \eqref{EqUnivJS}, let \((\E(1),s)|_{\PP_i}\) denote its restriction to \(\PP_i\subset \PP\).

\begin{theorem}\label{TheCasesJS}
Assume \(p_g(S)>0\). For each \(i = 0,\ldots,s\), (at least) one of the following statements holds:
\begin{enumerate}[i)]
\item The family of Joyce-Song pairs \((\E(1),s)|_{\PP_i}\) on \(\PP_i\times X \rightarrow \PP_i\) is fibrewise stable;
\item The virtual class \([\PP_i]^\vir\) vanishes.
\end{enumerate}
\end{theorem}

\begin{proof}
By the equations \eqref{EqCapEul}, and \eqref{EqCapGeenEul}, we may assume that we are in case ii) of Proposition~\ref{PropCases}. Hence for \(x\in \PP\), we have
\[E_x = E^0_x \oplus \ldots \oplus E^{-s}_x\]
with
\[\ch(E^0_x) = \ldots = \ch(E^{-s}_x)\,.\]
It is easy to see that the Joyce-Song pair \((\E_x,s_x)\) on \(X\) is stable, precisely when the first term of
\[s_x = \sum_i s_x^i \in H^0(X,\E_x) = H^0(S,E^0_x) \oplus \ldots \oplus H^0(S,E^{-s}_x)\]
is non-zero. It follows that the fibres of \((\E(1),s)|_{\PP_i}\) are stable if and only if \(i=0\).

Now let \(i>0\). I claim that we have
\[[M_\alpha^{n}]^\vir = \iota_* [M_\alpha^n]^\vir_\mathrm{loc}\]
for a cycle class \([M_\alpha^n]^\vir_\mathrm{loc}\) on the closed subscheme
\begin{equation} \label{EqDegeneracyLocus}
\iota\colon M_\alpha^n(\sigma_i) \coloneqq \left\{ b\in M_\alpha^n  \middle | E^{-(i-1)}|_b \cong E^{-i}|_b \right\} \hookrightarrow M_\alpha^n \,.
\end{equation}
Consider the Cartesian square
\[
\begin{tikzcd}
\PP_i(\sigma_i) \arrow[r, "\iota'"] \arrow[d, "p_i"] & \PP_i \arrow[d, "p_i"] \\
M_\alpha^n(\sigma_i) \arrow[r, "\iota"] & M_\alpha^n \,.
\end{tikzcd}
\]
Then we have by the claim
\begin{align*}
p_i^*[M_\alpha^n]
	& = p_i^*\iota_* [M_\alpha^n]^\vir_\mathrm{loc} \\
	& = \iota'_* p_i^* [M_\alpha^n]^\vir_\mathrm{loc} \,.
\end{align*}

On the other hand, recall that we write \(E_m^{-(i)} = E^{-(i)} \otimes \O_S(mH)\) for fixed \(m \gg 0\), so we have by cohomology-and-basechange we have
\begin{align*}
\iota^{\prime*}p_i^*\pi_{S*}E_m^{-(i-1)}(1)
	& = p_i^*\pi_{S*}(\iota\times id_S)^*E_m^{-(i-1)}(1) \\
	& = p_i^*\pi_{S*}(\iota\times id_S)^*E_m^{-i}(1) \\
	& = \iota^{\prime*}p_i^*\pi_{S*}E_m^{-i}(1) \\
	& = \iota^{\prime*}T_{\PP_i/M_\alpha^n} + \O \,,
\end{align*}
and hence by \eqref{EqCapEul} and the projection formula for \(\iota'\) we find
\begin{align*}
[\PP_i]^\vir
	& = p_i^*[M_\alpha^n]^\vir \cap e\left(\left(p_i^*\pi_{S*}E_m^{-(i-1)}(1)\right)^\vee\right) \\
	& = \iota'_*\left(p_i^*[M_\alpha^n]^\vir_\mathrm{loc} \cap \iota^{\prime*}e\left(\left(p_i^*\pi_{S*}E_m^{-(i-1)}(1)\right)^\vee\right)\right) \\
	& = \iota'_*\left(p_i^*[M_\alpha^n]^\vir_\mathrm{loc} \cap e\left(\iota^{\prime*}T_{\PP_i/M_\alpha^n}^\vee + \O \right)\right) \\
	& = 0 \,.
\end{align*}

We will prove the claim using the cosection \eqref{EqCosections}. For \(b\in M_\alpha^n\) we have \(\ch(E^{-i}|_b) = \ch(E^{-(i-1)}|_b)\), and hence
\[\Hom_{\O_S}(E^{-i}|_b, E^{-(i-1)}|_b ) = 0 \quad \text{or} \quad E^{-i}|_b \cong E^{-(i-1)}|_b\,,\]
since \( E^{-(i-1)}|_b\) and \(E^{-i}|_b\) are torsion free and of rank one.
By the exact sequence \eqref{EqCosections} and Serre duality, it follows that the cosection
\[ob_{M_\alpha^n} \xrightarrow{\sigma_i} R\pi_{S*}\O_S \]
is surjective on the complement of the locus \eqref{EqDegeneracyLocus}. The claim follows by the work of Kiem and Li \cite{KL}.
\end{proof}

\section{Comparison to the VW moduli space}
We need the following lemma (compare to \cite[Proposition~5.1]{TT2}).
\begin{lemma}\label{LemmaCanonES}
Let \((\E,s)\in (\P^\perp)^{\CC^*} \) be a \(\CC^*\) fixed Joyce-Song pair. Then \(\E\) has a \emph{canonical} \(\CC^*\) equivariant structure.
\end{lemma}
\begin{proof}
We will write
\[
\begin{tikzcd}
\CC^* \times X \arrow[r, shift left, "\rho"] \arrow[r, shift right, swap, "\pr_2"] & X
\end{tikzcd}
\]
for the action morphism and the projection respectively. Since \((\E,s)\) is \(\CC^*\) fixed, there exists an isomorphism \(\varphi\colon \rho^*\E \rightarrow \pr_2^*\E\) that fits in the diagram
\[
\begin{tikzcd}
\O(-m) \arrow[r, "\rho^*s"]\arrow[dr, swap, "\pr_2^*s"] & \rho^* \E \arrow[d,"\varphi"] \\
 & pr_2^*\E \,.
\end{tikzcd}
\]
By the stability of \((\E,s)\), we see that \(\varphi|_{\{1\}\times X} \colon \E\rightarrow\E\) is the identity. The argument of e.g.\ \cite[Proposition 4.4]{K} shows that \(\varphi\) defines an equivariant structure on \(\E\), which is unique up to a twist by a character \(\t^z \in \mathrm{Aut}(\CC^*) \cong \ZZ\). Now \(\varphi\) induces a \(\CC^*\)-action on \(E = q_* \E\). After multiplying \(\varphi\) by a power of \(\t\), we may assume that the highest weight occurring in the weight space decomposition of \(E\) is zero. We call \(\varphi\) the canonical equivariant structure on \(\E\).
\end{proof}

Fix a charge
\(\gamma = (r, c_1, c_2) \in H^\mathrm{even}(S)\),
and classes \(\alpha_0,\ldots,\alpha_s \in H^2(S,\ZZ)\) and \(n_0,\ldots,n_s \in H^4(S,\ZZ) = \ZZ\) with
\begin{equation}\label{EqCompatible}
r = s+1\,,\quad c_1 = \alpha_0 + \ldots + \alpha_s\,, \quad c_2 = n_0 + \ldots + n_s + \sum_{i<j} \alpha_i \alpha_j\,.
\end{equation}
Let \(L\) be a vector bundle on \(S\) with \(c_1(L) = c_1\), as in the introduction, write
\[\gamma^\perp = (r, L , c_2)\,.\]
Recall that we have defined
\[\P^\perp_{1^r} \subset (\P^\perp)^{\CC^*}\]
as the (open and closed) locus of Joyce-Song pairs \((\E,s)\) with
\[\rk(E^{-i}) = 1\,, \quad i = 0,\ldots, s\]
in the canonical weight space decomposition
\[E = q_* \E = E^0 \oplus \ldots \oplus E^{-s}\,.\]
induced by Lemma~\ref{LemmaCanonES}. Now define an open and closed subscheme
\[\P_\alpha^n \subset \P^\perp_{1^r}\]
by the rule
\[c(E^{-i}) =  1 + \alpha_i + n_i \,, \quad i=0,\ldots,s \,.\]

In the terminology of the introduction, the family \(\E\) on \(M_\alpha^n \times X \rightarrow M_\alpha^n\) constructed in Section~\ref{SecConstruction} is a family of sheaves of type \(\gamma^\perp\) (the centre of mass property is given by the fact that the corresponding Higgs field \(\phi\colon E \rightarrow E \otimes \omega_S\) is trace-free). It follows that the family of Joyce-Song pairs \((\E(1), s)\) on \(\PP\times X \rightarrow \PP\) defines a morphism
\begin{equation}\label{EqComparison}
\PP \supset \PP^{\mathrm{stab}} \rightarrow \P^\perp = \P^\perp_\gamma(m)
\end{equation}
on its stable locus.

\begin{proposition}\label{PropComparison}
The map \eqref{EqComparison} restricts to an isomorphism
\[(\PP^\mathrm{stab})^{\CC^*} \cong \P_\alpha^n\,.\]
\end{proposition}
\begin{proof}
C.f.\ \cite[Proposition~3.9]{GSY2}, this follows directly by comparing functors of points.
\end{proof}

\begin{lemma}\label{LemmaComparison}
For each \(i = 0,\ldots, s\) we have \(\PP_i \subset \PP^\mathrm{stab}\) or \(\PP_i \cap \PP^\mathrm{stab} = \emptyset\).
\end{lemma}
\begin{proof}
Following the proof of \cite[Lemma 3.2]{L}, the stability of \((\E_x,s_x)\) for \(x\in \PP_i\) is given by a numerical condition on \(\alpha_0,\ldots,\alpha_s\) and \(n_0,\ldots,n_s\), which is constant on \(\PP_i\).
\end{proof}

\begin{corollary}
\(\P^\perp_{1^r}\) is a disjoint union of schemes of the form \(\PP_i\).
\end{corollary}

We will show that the isomorphism of Proposition~\ref{PropComparison} respects the virtual structure. Let
\[(\mathscr{E},s)\]
denote the universal Joyce-Song pair on
\[\pi_X\colon \P^\perp \times X \rightarrow \P^\perp\,,\]
and consider the class
\[\mathscr{I}^\bullet = [\O \xrightarrow{s} \mathscr{E}_m] \in \D^b(\P^\perp \times X)\,.\]
There exists a splitting \cite{TT2}
\begin{multline}\label{EqDefPerp}
\RHom_{\pi_X}(\mathscr{I}^\bullet,\mathscr{I}^\bullet) = \RHom_{\pi_X}(\mathscr{I}^\bullet,\mathscr{I}^\bullet)_\perp \\
\oplus R\pi_{X*}\O_X \oplus R^{\geq1}\pi_{S*}\O_S \oplus R^{\leq1}\pi_{S*}\omega_S \otimes \t[-1]
\end{multline}
of objects in \(\D^b(\P^\perp)\), and by definition the perfect obstruction theory on \(\P^\perp\) is given by the dual of
\[\RHom_{\pi_X}(\mathscr{I}^\bullet,\mathscr{I}^\bullet)_\perp[1]\,.\]
In particular, the virtual tangent bundle of \(\P^n_\alpha\) is given by
\[\left( - \RHom_{\pi_X}(\mathscr{I}^\bullet,\mathscr{I}^\bullet)_\perp \Big |_{\P^n_\alpha} \right)^{\mathrm{fix}} \in K^0(\P^n_\alpha)\,.\]

Similarly, write
\[I^\bullet = [\O_{X\times \PP} \xrightarrow{s} \E_m(1)] \in \mathcal{D}^b(\PP \times X)\,,\]
 and let the object
\[\RHom_{\pi_X}(I^\bullet,I^\bullet)_\perp  \in \mathcal{D}^b(\PP)\]
be given by the rule \eqref{EqDefPerp} above.
Finally, let \(\RHom_{\pi_X}(\E,\E)_\perp \in\D^b(M_\alpha^n)\) be the object considered in \cite{TT} satisfying
\[\RHom_{\pi_X}(\E,\E) = \RHom_{\pi_X}(\E,\E)_\perp \oplus R\pi_{S*}\O_S \oplus R\pi_{S*}\omega_S \otimes \t[-1]\,.\]

By definition, pull-back along the morphism \eqref{EqComparison} identifies the restrictions to \(\PP^\mathrm{stab}\) of the Joyce-Song pairs \((\mathscr{E},s)\) and \((\E(1),s)\), and hence we have
\[ \mathscr{I}^\bullet |_{\PP^\mathrm{stab}} \cong I^\bullet |_{\PP^\mathrm{stab}}\]
and
\[\RHom_{\pi_X}(\mathscr{I}^\bullet,\mathscr{I}^\bullet)_\perp|_{\PP^\mathrm{stab}}
\cong
\RHom_{\pi_X}(I^\bullet,I^\bullet)_\perp|_{\PP^\mathrm{stab}}\,.\]

Following \cite[Proposition~6.8]{TT2}, we find
\begin{equation}\label{EqCompareVirTan}
\begin{split}
\RHom_{\pi_X}(I^\bullet,I^\bullet)_\perp
	={}& \RHom_{\pi_X}(I^\bullet,I^\bullet)_0 - R^{\geq1}\pi_{S*}\O_S + R^{\leq1}\pi_{S*}\omega_S \otimes \t \\
	={}& p^*\RHom_{\pi_X}(\E,\E) - R\pi_{S*}\O_S + R\pi_{S*}\omega_S \otimes \t \\
	&  - \pi_{X*}\E_m(1) + (\pi_{X*}\E_m(1))^\vee\otimes \t + \O_\PP - \O_\PP\otimes\t \\
	={}& p^*\RHom_{\pi_X}(\E,\E)_\perp  - T_{\PP/M_\alpha^n} + T_{\PP/M_\alpha^n}^\vee \otimes \t
\end{split}
\end{equation}
in \(K^0_{\CC^*}(\PP)\).
The virtual tangent bundle of \(M_\alpha^n\) is precisely the \(\CC^*\) fixed part of \(-\RHom_{\pi_S}(\E,\E)_\perp\) (use \cite[Corollary 2.26]{TT} to compare \(-\RHom_{\pi_S}(\E,\E)_\perp\) with \eqref{EqPOTM}). If follows by \eqref{EqObstructionPi} that the isomorphism of Proposition~\ref{PropComparison} identifies the virtual tangent bundles of \(\P^n_\alpha\) and \((\PP^\mathrm{stab})^{\CC^*}\). By Siebert's formula \cite{Sie} and its K-theoretic analogue \cite[Theorem~4.2]{T2}, we find that virtual classes and virtual structure sheaves agree. Moreover, by Theorem~\ref{TheCasesJS}, we have the following proposition
\begin{proposition}\label{PropVirEqStab}
We have equalities
\begin{align*}
[\P_\alpha^n]^\vir
	& = [(\PP^\mathrm{stab})^{\CC^*}]^\vir \\
	& = [\PP^{\CC^*}]^\vir
\intertext{where we have identified}
A_*(\P_\alpha^n)
	& = A_*((\PP^\mathrm{stab})^{\CC^*})\\
	& \subset A_*(\PP^{\CC^*})\,.
\end{align*}
Similarly, we have
\begin{align*}
\O_{\P_\alpha^n}^\vir
	& =\O_{(\PP^\mathrm{stab})^{\CC^*}}^\vir \\
	& =  \O_{\PP^{\CC^*}}^\vir
\intertext{in}
K_0(\P_\alpha^n)
	& = K_0((\PP^\mathrm{stab})^{\CC^*})\\
	& \subset K_0(\PP^{\CC^*})\,.
\end{align*}
\end{proposition}

\section{The unrefined case}\label{SecUnref}
Theorem \ref{TheCasesJS} (or its consequence Proposition~\ref{PropVirEqStab}) allows us compute the contribution of \(\P_\alpha^n\) to the Vafa-Witten invariant as a integral over the virtual class of \(\PP\), which might contain unstable Joyce-Song pairs. We will use this to prove the unrefined part of Theorem~\ref{ResultA}.

Let
\begin{align*}
N^\vir_{\PP^{\CC^*}\!\!/\PP}
	& = \left(T_\PP |_{\PP^{\CC^*}}\right)^\mathrm{mov} \\
	& = \left(\left(T_{\PP/M_\alpha^n} - T_{\PP/M_\alpha^n}^\vee \otimes \t \right)\middle|_{\PP^{\CC^*}}\right)^\mathrm{mov} 
\end{align*}
be the virtual normal bundle to \(\PP^{\CC^*}\) in \(\PP\), and let \(N^\vir_{\P_\alpha^n/ \P^\perp}\) be the virtual normal bundle to \(\P_\alpha^n\) in \(\P^\perp\). We will also write
\[N_\alpha^n = \left(-\RHom_{\pi_X}(\E,\E)_\perp \right)^\text{mov}\]
for the class in \(K^0_{\CC^*}(M_\alpha^n)\).
Note that by \eqref{EqCompareVirTan} we have
\begin{align*}
N^\vir_{\P_\alpha^n/ \P^\perp}
	& = \left(-\RHom_{\pi_X}(I^\bullet , I^\bullet)_\perp \vert_{(\PP^\mathrm{stab})^{\CC^*}}\right)^\text{mov} \\
	& = \left( p^*N_\alpha^n + N^\vir_{\PP^{\CC^*}\!\!/\PP}\middle)\right\vert_{(\PP^\mathrm{stab})^{\CC^*}}
\end{align*}
using the identification
\[\P_\alpha^n \cong (\PP^\mathrm{stab})^{\CC^*}\]
of Proposition~\ref{PropComparison}.
\begin{remark}
If the family of sheaves \(\E\) on \(M_\alpha^n \times X \rightarrow M_\alpha^n\) is Gieseker \emph{stable}, \(M_\alpha^n\) is an open and closed subscheme of the fixed locus
\((\N^\perp)^{\CC^*}\) of the moduli space \(\N^\perp\) (defined in \cite{TT}) of stable compactly supported sheaves on \(X\) of type \(\gamma^\perp\). In this case, the class \(N_\alpha^n\) is the virtual normal bundle to \(M_\alpha^n\) in \(\N^\perp\).
\end{remark}

By virtual \(\CC^*\) localisation \cite{GP} and Proposition~\ref{PropVirEqStab} we have
\begin{align*}
\int_{[\PP]^\vir} \frac{1}{p^*e(N_\alpha^n)}
	& = \int_{[\PP^{\CC^*}]^\vir} \frac{1}{p^*e(N_\alpha^n) e(N_{\PP^{\CC^*}\!\!/\PP}^\vir)} \notag \\
	& =  \int_{[\P^n_\alpha]^\vir} \frac{1}{e\left(N^\vir_{\P^n_\alpha/\P^\perp}\right)}\,.
\end{align*}
Note that the last line is exactly the contribution of the open and closed locus
\[\P_\alpha^n \subset (\P^\perp)^{\CC^*} = (\P^\perp_\gamma(m))^{\CC^*}\]
to the integral of Conjecture-Definition~\ref{ConDefUnref}. On the other hand, by Corollary~\ref{CorVirClass} and the projection formula, we have
\begin{align*}
\int_{[\PP]^\vir} \frac{1}{p^*e(N_\alpha^n)} 
	& = \int_{p^*[M_\alpha^n]^\vir}\frac{ e\!\left(T_{\PP/M_\alpha^n}^\vee\right)}{p^*e(N_\alpha^n)}  \\
	& = \int_{[M_\alpha^n]^\vir} \frac{p_*e\!\left(T_{\PP/M_\alpha^n}^\vee\right)}{e(N_\alpha^n)} \\
	& = (-1)^{\chi(\gamma(m))-1} \chi(\gamma(m)) \int_{[M_\alpha^n]^\vir} \frac{1}{e(N_\alpha^n)} \,.
\end{align*}
It follows that the vertical contribution \(\VW_{\gamma}^\vertical\) to the (unrefined) Vafa-Witten invariant (defined in Section~\ref{SubsectionResults}), is given by
\begin{equation}\label{EqWellDefined}
\VW_\gamma^\vertical = \sum_{\alpha,n} \int_{[M_\alpha^n]^\vir} \frac{1}{e(N_\alpha^n)}\,,
\end{equation}
where the sum is taken over classes
\[\alpha_0,\ldots,\alpha_s \in H^2(S,\ZZ)\,, \quad n_0,\ldots,n_s \in H^4(S,\ZZ)\]
satisfying
\[c_1 = \alpha_0 + \ldots + \alpha_s\,, \quad c_2 = n_0 + \ldots + n_s + \sum_{i<j} \alpha_i \alpha_j\,.\]
This proves the unrefined part of Theorem~\ref{ResultA}.

\begin{remark}
The computation in this section is used in \cite[Proposition 6.8]{TT2} to show that the different definitions of the Vafa-Witten invariant given in \cite{TT} and \cite{TT2} agree in the case that $E$ is stable.
\end{remark}

\section{The integrals}\label{SecIntegrals}
Recall from Sections~\ref{SecConstruction} and \ref{SectionVirtualClass} that we write
\begin{align*}
\beta_1
	& = \alpha_1-\alpha_0 + c_1(\omega_S) \\
	& \vdotswithin{=} \\
\beta_{s}
	& = \alpha_{s} - \alpha_{s-1} + c_1(\omega_S) \,,
\end{align*}
for the classes in \(H^2(S,\ZZ)\), and
\[\eta\colon M_\alpha^n \rightarrow S_\beta^n\]
for the surjective \'etale morphism to the nested Hilbert scheme. Note that since
\[\RHom_{\pi_X}(\E,\E)_\perp = \RHom_{\pi_X}(\E\otimes M,\E\otimes M)_\perp\]
for any line bundle \(M\) on \(S\), the class \(N_\alpha^n\) is invariant under the action of \(\Pic_0(S)[s+1]\) on \(M_\alpha^n\) (see Section~\ref{SectionVirtualClass}). It follows that it descents along \(\eta\) to a class \(N_\beta^{[n]}\) in \(K^0_{\CC^*}(S_\beta^{[n]})\), so we have
\begin{align}\label{EqIntegralsEta}
\int_{[M_\alpha^n]^\vir}\frac{1}{e(N_\alpha^n)}
	& = \int_{\eta_*[M_\alpha^n]^\vir}\frac{1}{e(N_\beta^{[n]})} \\
	& = \#\Pic_0(S)[s+1]\int_{[S_\beta^{[n]}]^\vir}\frac{1}{e(N_\beta^{[n]})}\,. \notag
\end{align}

Recall from \cite{TT}, that we have
\begin{equation}\label{EqSpectralObstruction}
\RHom_{\pi_X}(\E,\E)_\perp = \RHom_{\pi_S}(E,E)_0 - \RHom_{\pi_S}(E,E\otimes\omega_S\otimes \t)_0
\end{equation}
in \(K^0_{\CC^*}(M_\alpha^n)\). Choose line bundles
\[L_i\in \Pic_{\alpha_i}(S) \,, \quad i = 0,\ldots,s \]
and write
\[E^{[n]}_{\vectorL} = \bigoplus_{i=0}^s \I^{[n_i]} \otimes L_i \otimes \t^{-i}\]
for the family of sheaves on
\[\pi_S\colon S^{[n_0]} \times \cdots \times S^{[n_s]} \times S \rightarrow S^{[n_0]} \times \cdots \times S^{[n_s]} \,.\]
Finally, define a class
\[N_\vectorL^{[n]} \coloneqq \Big( \RHom_{\pi_S}\left(E^{[n]}_{\vectorL},E^{[n]}_{\vectorL}\otimes\omega_S\otimes \t\right)_0 - \RHom_{\pi_S}\left(E^{[n]}_{\vectorL},E^{[n]}_{\vectorL}\right)_0 \Big)^\mathrm{mov}\]
in \(K^0_{\CC^*}(S^{[n_0]} \times \cdots \times S^{[n_s]})\).

Write
\[\rho\colon S_\beta^{[n]} \rightarrow S^{[n_0]} \times \cdots \times S^{[n_s]}\]
for the natural morphism. By \cite[Theorem 5.6 and (the discussion before) Theorem 6]{GT2}, we have
\begin{align} \label{AlignCompareLaa1}
\VW_\beta^{[n]}
	\coloneqq{}& \int_{[S_\beta^{[n]}]^\vir}\frac{1}{e(N_\beta^{[n]})} \notag \\
	={}& \int_{\rho_*[S_\beta^{[n]}]^\vir}\frac{1}{e(N_\vectorL^{[n]})} \notag \\
	={}& \SW(\beta) \int_{S^{[n_0]} \times \cdots \times S^{[n_s]}} \frac{\mathsf{co}_\vectorL^{[n]}}{e(N_\vectorL^{[n]})}
\end{align}
where
\[\SW(\beta) = \SW(\beta_1)\cdots \SW(\beta_s)\]
is the product of Seiberg-Witten invariants of the classes \(\beta_i \in H^2(S,\ZZ)\), and
\begin{align*}
\mathsf{co}_\vectorL^{[n]}
	\coloneqq{} & c_\mathrm{top}\left(\bigoplus_{i=1}^s \RHom_{\pi_S}\left(\I^{[n_{i-1}]} \otimes L_{i-1}, \I^{[n_i]} \otimes L_i \otimes \omega_S\right)_0 \right)\\
	\in{} & A^{2|n| - n_0 - n_s}(S^{[n_0]} \times \cdots \times S^{[n_s]})
\end{align*}
is a Carlsson-Okounkov type class \cite{CO,GT2}.
The integral \eqref{AlignCompareLaa1} is precisely the one of \cite[Equation~4.4]{L}, which computes vertical contributions to the Vafa-Witten invariant in the case that stable = semistable.

\begin{proof}[Proof of Theorem~\ref{ResultB} (unrefined case)]
Fix a rank \(r = s+1\). We will use the following convention.
For \(\beta = (\beta_1,\ldots,\beta_s)\in H^2(S,\ZZ)^s\), we will write
\begin{align*}
\beta^\vee
	& = (\beta^1,\ldots,\beta^s) \\
	& = (c_1(\omega_S) - \beta_1, \ldots, c_1(\omega_S) - \beta_s)\,.
\end{align*}
In particular we have
\begin{align*}
\SW(\beta^\vee)
	& = \SW(\beta^1) \cdots \SW(\beta^s) \\
	& = (-1)^{s \cdot \chi(\O_S)} \SW(\beta)\,.
\end{align*}
By \cite[Propositions 6.3 and 7.2]{L}, and the proof of Theorem A of loc.cit., there exist universal Laurent series (i.e.\ depending only on \(r\))
\[A,B,C_{ij} \in \QQ(\!(q^\frac{1}{2r})\!)\,, \quad 1\leq i \leq j <r\]
such that
\begin{align}\label{EqUnivSeriesVW}
q^{d(\beta)} \sum_n \VW_\beta^{[n]} q^{|n|}
	& =  \SW(\beta^\vee)A^{\chi(\O_S)} B^{K_S^2}  \prod_{i\leq j} C_{ij}^{\beta^i\beta^j}
\end{align}
for any surface \(S\) and classes \(\beta = (\beta_1,\ldots,\beta_s) \in H^2(S,\ZZ)\).
We have use a normalizing factor \(q^{d(\beta)}\) given by
\[d(\beta) = -\sum_{1\leq i,j < r} \frac{i(r-j)}{2r}\beta^i\beta^j\,.\]

Now fix a surface \(S\) with \(p_g(S)>0\), a class \(c_1 \in H^2(S,\ZZ)\) and a line bundle \(L\) on \(S\) with \(c_1(L) = c_1\). For an \(r\)-tuple
\[\alpha = (\alpha_0,\ldots,\alpha_s)\in (H^2(S,\ZZ))^r\,,\]
write
\[\beta(\alpha) = (\alpha_1 - \alpha_0 + c_1(\omega_S), \ldots, \alpha_s-\alpha_{s-1} + c_1(\omega_S))\,.\]
Also write
\[c_2(\alpha,n) = n_0 + \ldots + n_s + \sum_{i<j} \alpha_i\alpha_j\]
for \(n \in (\ZZ_{\geq0})^r\) and \(\alpha\) as above.

By \cite[Lemma 2.6]{L}, we have have
\[c_2(\alpha,n) = n_0 + \ldots + n_s + d(\beta(\alpha)) + \frac{r-1}{2r}c_1^2.\]
Following the proof of \cite[Theorem~A]{L}, we can use this to rewrite the sum
\[\sum_{\alpha,n} \VW_{\beta(\alpha)}^{[n]} q^{c_2(\alpha,n)}\]
over \(\alpha \in H^2(S,\ZZ)^r\) and \(n \in (\ZZ_{\geq0})^r\) satisfying
\[\sum_{i=0}^{r-1} \alpha_i = c_1\,,\]
as a sum
\[\#H^2(S,\ZZ)[r] \, \sum_{\beta,n} \VW_{\beta}^{[n]} q^{|n| + d(\beta) + \frac{r-1}{2r}c_1^2}\]
over \(\beta \in H^2(S,\ZZ)^{r-1}\) and \(n \in (\ZZ_{\geq0})^r\)  satisfying
\begin{equation} \label{EqSumBeta}
\sum_{i=1}^{r-1} i \beta^i \equiv c_1 \mod rH^2(S,\ZZ)\,.
\end{equation}
(To see this, note that the equations
\[\beta(\alpha) = \beta \, , \quad \sum_{i=0}^s \alpha_i = c_1\]
have precisely \(\#H^2(S,\ZZ)[r]\) solutions \(\alpha\) for any \(\beta\) and \(c_1\) satisfying \eqref{EqSumBeta}.) By \eqref{EqWellDefined}, \eqref{EqIntegralsEta} and \eqref{EqUnivSeriesVW} it follows that
\begin{align*}
\sum_{c_2} \VW_{(r,L,c_2)}^\vertical q^{c_2}
	& = \sum_{\alpha,n} \int_{[M_\alpha^n]^\vir} \frac{1}{e(N_\alpha^n)} q^{c_2(\alpha,n)} \\
	& = \#\Pic_0(S)[r] \, \sum_{\alpha,n} \VW_{\beta(\alpha)}^{[n]} q^{c_2(\alpha,n)} \\
	& = \#\Pic_0(S)[r] \cdot \# H^2(S,\ZZ)[r] \, \sum_{\beta,n} \VW_{\beta}^{[n]} q^{|n| + d(\beta) + \frac{r-1}{2r}c_1^2} \\
	& = \#\Pic(S)[r] \, q^{\frac{r-1}{2r}c_1^2} \,A^{\chi(\O_S)} \, B^{K_S^2} \sum_{\beta} \SW(\beta^\vee) \, \prod_{i\leq j} C_{ij}^{\beta^i\beta^j} \,. \qedhere
\end{align*}

\end{proof}

\section{Refined invariants}
We will briefly discuss refined invariants. The proofs of Theorems \ref{ResultA} and \ref{ResultB} are perfectly analogous to the discussion of Sections \ref{SecUnref} and \ref{SecIntegrals}. Using the equivariant structure \eqref{EqEquivStrucTan}, we have by Corollary~\ref{CorVirStruc},
\[\O_\PP^\vir = p^* \O_{M_\alpha^n}^\vir \otimes \Lambda^\bullet (T_{\PP/M_\alpha^n} \otimes \t^{-1}) \in K_0^{\CC^*}\!(\PP)\,.\]
Choose a square root
\[K^\frac{1}{2} \coloneqq \det \left(-\RHom_{\pi_X}(I^\bullet , I^\bullet)_\perp ^\vee \right)^\frac{1}{2}\,.\]
Note that by equations \eqref{EqCompareVirTan} and \eqref{EqSpectralObstruction}, we can take 
\begin{align*}
K^\frac{1}{2}
	& = p^*\det\left(\RHom_{\pi_X}(\E,\E)_\perp\right)^\frac{1}{2} \otimes \omega_{\PP/M_\alpha^n} \otimes \t^\frac{\chi(\gamma(m))-1}{2}
\end{align*}
with
\[\det(\RHom_{\pi_X}(\E,\E)_\perp)^\frac{1}{2} \coloneqq \det \left(\RHom_{\pi_S}(E,E)_0\otimes \t^{-\frac{1}{2}}\right)\,.\]

By the K-theoretic virtual localisation formula \cite{Qu}, we have, again using \ref{PropVirEqStab}:

\begin{align*}
\chi_t\left(\PP, \frac{\O_\PP^\vir \otimes K^\frac{1}{2}}{p^*\Lambda^\bullet (N_\alpha^n)^\vee } \right)
	& = \chi_t\left(\PP^{\CC^*}, \frac{\O_{\PP^{\CC^*}}^\vir \otimes K^\frac{1}{2}}{p^*\Lambda^\bullet (N_\alpha^n)^\vee \otimes
		\Lambda^\bullet (N_{\PP^{\CC^*}\!\!/\PP}^\vir )^\vee}\right)\\
	& = \chi_t\left(\P^n_\alpha, \frac{\O_{\P^n_\alpha}^\vir \otimes K^\frac{1}{2}}{\Lambda^\bullet (N^\vir_{\P^n_\alpha/\P^\perp})^\vee } \right) \,.
\end{align*}

On the other hand, he have
\begin{align*}
\chi_t\left(\PP, \frac{\O_\PP^\vir \otimes K^\frac{1}{2}} {p^*\Lambda^\bullet (N_\alpha^n)^\vee } \right)
	& = \chi_t\left(M_\alpha^n, p_*\left(\frac{\O_{\PP}^\vir \otimes K^\frac{1}{2}} {p^*\Lambda^\bullet (N_\alpha^n)^\vee } \right)\right) \\
	& = \chi_t\Bigg(M_\alpha^n, \frac{\O_{M_\alpha^n}^\vir \otimes  \det(\RHom_{\pi_X}(\E,\E)_\perp)^\frac{1}{2} }
		{\Lambda^\bullet (N_\alpha^n)^\vee } \\
	& \quad \otimes p_*\left( \Lambda^\bullet (T_{\PP/M_\alpha^n} \otimes \t^{-1}) \otimes \omega_{\PP/M_\alpha^n} \right) \otimes \t^\frac{\chi(\gamma(m))-1}{2} \Bigg) \\
	& = \chi_t\Bigg(M_\alpha^n, \frac{\O_{M_\alpha^n}^\vir \otimes  \det(\RHom_{\pi_X}(\E,\E)_\perp)^\frac{1}{2} }
		{\Lambda^\bullet (N_\alpha^n)^\vee } \Bigg)\\
	& \quad \times (-1)^{\chi(\gamma(m))} [\chi(\gamma(m))]_t
\,.
\end{align*}
Here we have used
\begin{align*}
p_*\left(\Lambda^\bullet (T_{\PP/M_\alpha^n} \otimes \t^{-1})\otimes \omega_{\PP/M_\alpha^n}  \right) \otimes \t^\frac{\chi(\gamma(m))-1}{2}
	& =  (-1)^{\chi(\gamma(m))} \frac{p_* \Lambda^\bullet (\Omega_{\PP/M_\alpha^n} \otimes \t)} {\t^{\frac{\chi(\gamma(m))-1}{2}}} \\
	& =  (-1)^{\chi(\gamma(m))} \frac{\t^{\chi(\gamma(m))-1} + \ldots + 1}{ \t^\frac{\chi(\gamma(m))-1}{2}} \\
	& =  (-1)^{\chi(\gamma(m))} [\chi(\gamma(m))]_\t
\,.
\end{align*}
As in the unrefined case, we conclude that the vertical contribution to the refined Vafa-Witten invariant is well defined, and c.f.\ \eqref{EqWellDefined} given by a sum
\[\VW_{\gamma}^\vertical(t) = \sum_{\alpha,n} \chi_t\Bigg(M_\alpha^n, \frac{\O_{M_\alpha^n}^\vir \otimes  \det(\RHom_{\pi_X}(\E,\E)_\perp)^\frac{1}{2}}{\Lambda^\bullet (N_\alpha^n)^\vee } \Bigg)\,.\]
This finishes the proof of Theorem~\ref{ResultA}.

By the virtual Riemann-Roch formula \cite{CK,FG}, we can compute the contribution of \(M_\alpha^n\) to the refined Vafa-Witten invariant by the integral
\[\left[\int_{[M_\alpha^n]^\vir} \frac{\ch\left( \det(\RHom_{\pi_X}(\E,\E)_\perp)^\frac{1}{2}\right)}{\ch\big(\Lambda^\bullet (N_\alpha^n)^\vee \big)}  \Td \left( T_{M_\alpha^n}\right) \right]_{\ch{\t} = t}\,,\]
where
\[T_{M_\alpha^n} = \left(-\RHom_{\pi_X}(\E,\E)_\perp\right)^\mathrm{fix}\]
denotes the virtual tangent bundle of \(M_\alpha^n\). The discussion of Section~\ref{SecIntegrals} now literally applies to this integral: the integrant descents to a class in \(A^*_{\CC^*}(S_\beta^{[n]})\), which we can integrate over
\[\eta_*[M_\alpha^n]^\vir = \#\Pic_0(S)[r] \cdot [S_\beta^{[n]}]^\vir \,.\]
Again, we can use the results of \cite{GT2} to rewrite the resulting integral as an integral over the product of Hilbert schemes of points
\[S^{[n_0]} \times \cdots \times S^{[n_s]}\,.\]
In fact, after taking out the factor \(\#\Pic_0(S)[r]\), it is precisely the one given in \cite[Equation 4.5]{L}. Copying the proof of the unrefined case given in Section \ref{SecIntegrals}, Theorem~\ref{ResultB} now follows from \cite[Propositions 6.5 and 7.5]{L}.

\bibliographystyle{mijnhamsalpha}
\bibliography{proefschrift}{}

\providecommand{\bysame}{\leavevmode\hbox to3em{\hrulefill}\thinspace}
\providecommand{\href}[2]{#2}
\providecommand{\eprint}{\begingroup \urlstyle{rm}\Url}
\begin{thebibliography}{Tho18b}

\bibitem[CFK09]{CK}
I.~Ciocan-Fontanine and M.~Kapranov, \emph{Virtual fundamental classes via
  dg-manifolds}, Geom. Topol. \textbf{13} (2009), no.~3, 1779--1804.

\bibitem[CO12]{CO}
E.~Carlsson and A.~Okounkov, \emph{Exts and vertex operators}, Duke Math. J.
  \textbf{161} (2012), no.~9, 1797--1815.

\bibitem[FG10]{FG}
B.~Fantechi and L.~G{\"o}ttsche, \emph{Riemann-{R}och theorems and elliptic
  genus for virtually smooth schemes}, Geom. Topol. \textbf{14} (2010), no.~1,
  83--115.

\bibitem[GP99]{GP}
T.~Graber and R.~Pandharipande, \emph{Localization of virtual classes}, Invent.
  Math. \textbf{135} (1999), no.~2, 487--518.

\bibitem[GSY17]{GSY2}
A.~Gholampour, A.~Sheshmani, and S.-T. Yau, \emph{Localized
  {D}onaldson-{T}homas theory of surfaces}, 2017, \eprint{arXiv:1701.08902}.

\bibitem[GSY18]{GSY}
\bysame, \emph{Nested {H}ilbert schemes on surfaces: {V}irtual fundamental
  class}, Journal of {D}ifferential {G}eometry (to appear). (2018),
  \eprint{arXiv:1701.08899}.

\bibitem[GT19]{GT2}
A.~Gholampour and R.~P. Thomas, \emph{Degeneracy loci, virtual cycles and
  nested {H}ilbert schemes {II}}, 2019, \eprint{arXiv:1902.04128}.

\bibitem[JS12]{JS}
D.~Joyce and Y.~Song, \emph{A theory of generalized {D}onaldson-{T}homas
  invariants}, Mem. Amer. Math. Soc. \textbf{217} (2012), no.~1020, iv+199.

\bibitem[KL13]{KL}
Y.-H. Kiem and J.~Li, \emph{Localizing virtual cycles by cosections}, J. Amer.
  Math. Soc. \textbf{26} (2013), no.~4, 1025--1050.

\bibitem[Koo11]{K}
M.~Kool, \emph{Fixed point loci of moduli spaces of sheaves on toric
  varieties}, Adv. Math. \textbf{227} (2011), no.~4, 1700--1755.

\bibitem[Laa18]{L}
T.~Laarakker, \emph{Monopole contributions to refined {V}afa-{W}itten
  invariants}, 2018, \eprint{arXiv:1810.00385}.

\bibitem[Qu18]{Qu}
F.~Qu, \emph{Virtual pullbacks in {$K$}-theory}, Ann. Inst. Fourier (Grenoble)
  \textbf{68} (2018), no.~4, 1609--1641.

\bibitem[Sie04]{Sie}
B.~Siebert, \emph{Virtual fundamental classes, global normal cones and
  {F}ulton's canonical classes}, Frobenius manifolds, Aspects Math., E36,
  Friedr. Vieweg, Wiesbaden, 2004, pp.~341--358.

\bibitem[Tho18a]{T}
R.~P. Thomas, \emph{Equivariant {K}-theory and refined {V}afa-{W}itten
  invariants}, 2018, \eprint{arXiv:1810.00078}.

\bibitem[Tho18b]{T2}
\bysame, \emph{A {K}-theoretic {F}ulton class}, 2018,
  \eprint{arXiv:1810.00079}.

\bibitem[TT17a]{TT}
Y.~Tanaka and R.~P. Thomas, \emph{Vafa-{W}itten invariants for projective
  surfaces {I}: stable case}, Jour. Alg. Geom (to appear) (2017),
  \eprint{arXiv:1702.08487}.

\bibitem[TT17b]{TT2}
\bysame, \emph{Vafa-{W}itten invariants for projective surfaces {II}:
  semistable case}, 2017, \eprint{arXiv:1702.08488}.

\end{thebibliography}

\end{document}